\documentclass[12pt, reqno]{amsart}

\usepackage[T1]{fontenc}
\usepackage{xcolor}
\usepackage{latexsym}
\usepackage{amssymb}
\usepackage{mathrsfs}
\usepackage{amsmath}
\usepackage{fancybox,color}
\usepackage{enumerate}
\usepackage[latin1]{inputenc}
\usepackage{eurosym}
\usepackage{url}
\usepackage{hyperref}
  \hypersetup{colorlinks=true, %
   citecolor=magenta,%
   linkcolor=magenta}
\newcommand{\kom}[1]{}
\renewcommand{\kom}[1]{{\bf [#1]}}

\newcommand\blfootnote[1]{%
  \begingroup
  \renewcommand\thefootnote{}\footnote{#1}%
  \addtocounter{footnote}{-1}%
  \endgroup
}

 \def\1{\raisebox{2pt}{\rm{$\chi$}}}

\newtheorem{theorem}{Theorem}[section]

\newtheorem{lemma}[theorem]{Lemma}
\newtheorem{proposition}[theorem]{Proposition}
\newtheorem{definition}[theorem]{Definition}
\newtheorem{remark}[theorem]{Remark}

 \def\1{\raisebox{2pt}{\rm{$\chi$}}}
 
\DeclareMathOperator\supp{supp}
\def\vint_#1{\mathchoice%
          {\mathop{\kern 0.2em\vrule width 0.6em height 0.69678ex depth -0.58065ex
                  \kern -0.8em \intop}\nolimits_{\kern -0.4em#1}}%
          {\mathop{\kern 0.1em\vrule width 0.5em height 0.69678ex depth -0.60387ex
                  \kern -0.6em \intop}\nolimits_{#1}}%
          {\mathop{\kern 0.1em\vrule width 0.5em height 0.69678ex
              depth -0.60387ex
                  \kern -0.6em \intop}\nolimits_{#1}}%
          {\mathop{\kern 0.1em\vrule width 0.5em height 0.69678ex depth -0.60387ex
                  \kern -0.6em \intop}\nolimits_{#1}}}
\def\vintslides_#1{\mathchoice%
          {\mathop{\kern 0.1em\vrule width 0.5em height 0.697ex depth -0.581ex
                  \kern -0.6em \intop}\nolimits_{\kern -0.4em#1}}%
          {\mathop{\kern 0.1em\vrule width 0.3em height 0.697ex depth -0.604ex
                  \kern -0.4em \intop}\nolimits_{#1}}%
          {\mathop{\kern 0.1em\vrule width 0.3em height 0.697ex depth -0.604ex
                  \kern -0.4em \intop}\nolimits_{#1}}%
          {\mathop{\kern 0.1em\vrule width 0.3em height 0.697ex depth -0.604ex
                  \kern -0.4em \intop}\nolimits_{#1}}}

\newcommand{\kint}{\vint}

\newcommand{\aveint}[2]{\mathchoice%
          {\mathop{\kern 0.2em\vrule width 0.6em height 0.69678ex depth -0.58065ex
                  \kern -0.8em \intop}\nolimits_{\kern -0.45em#1}^{#2}}%
          {\mathop{\kern 0.1em\vrule width 0.5em height 0.69678ex depth -0.60387ex
                  \kern -0.6em \intop}\nolimits_{#1}^{#2}}%
          {\mathop{\kern 0.1em\vrule width 0.5em height 0.69678ex depth -0.60387ex
                  \kern -0.6em \intop}\nolimits_{#1}^{#2}}%
          {\mathop{\kern 0.1em\vrule width 0.5em height 0.69678ex depth -0.60387ex
                  \kern -0.6em \intop}\nolimits_{#1}^{#2}}}

\newcommand{\diam}{\operatorname{diam}}

\numberwithin{equation}{section}
\allowdisplaybreaks

\definecolor{color1}{rgb}{0.309, 0.43,0.258}
\definecolor{color2}{rgb}{0.741, 0.502,0.743}
\definecolor{color3}{rgb}{0.580, 0.163,0.107}
\definecolor{color0}{rgb}{0.825, 0.201, 0.699}
\begin{document}

\title[]{On $W^{2,p}$-estimates for solutions of obstacle problems for fully nonlinear elliptic equations
  with oblique boundary conditions}

\author[Byun]{Sun-Sig Byun}
\address{Department of Mathematical Sciences, Seoul National University, 
Seoul 08826, Republic of Korea}
\email{byun@snu.ac.kr}

\author[Han]{Jeongmin Han}
\address{Department of Mathematical Sciences, Seoul National University, 
Seoul 08826, Republic of Korea}
\email{hanjm9114@snu.ac.kr}

\author[Oh]{Jehan Oh}
\address{Department of Mathematics, Kyungpook National University, 
Daegu 41566, Republic of Korea}
\email{jehan.oh@knu.ac.kr}

\blfootnote{S.-S. Byun was supported by NRF-2017R1A2B2003877. J. Han was supported by NRF-2019R1C1C1003844. J.Oh was supported by NRF-2020R1C1C1A01014904.}

%\date{\today}

\begin{abstract}
This paper concerns fully nonlinear elliptic obstacle problems with oblique boundary conditions. 
We investigate the existence, uniqueness and $W^{2,p}$-regularity results by finding approximate
non-obstacle problems with the same oblique boundary condition and
then making a suitable limiting process.
\end{abstract}

\maketitle

%\tableofcontents

\section{Introduction and main results}

This paper concerns the existence, uniqueness and regularity for viscosity solutions to the following obstacle problem with oblique boundary data
\begin{align} \label{eloblobs}
\left\{ \begin{array}{ll}
F(D^{2}u,Du,u,x) \le f & \textrm{in $\Omega, $}\\
(F(D^{2}u,Du,u,x) -f)(u-\psi) = 0 & \textrm{in $\Omega$, } \\
u \ge \psi & \textrm{in $\Omega$,} \\
\beta \cdot Du=0 & \textrm{on $\partial \Omega$} 
\end{array} \right.
\end{align}
for a given obstacle $\psi \in W^{2,p}(\Omega)$ satisfying
$\beta \cdot D\psi \ge 0 $ a.e. on $\partial \Omega$.
Here $\Omega$ is a bounded domain in $\mathbb{R}^{n}$ with its boundary $\partial \Omega \in C^{3}$,
$F$ is uniformly elliptic with constants $\lambda$ and $\Lambda$, i.e.,
$$ \lambda ||N|| \le F(M+N,q,r,x)-F(M,q,r,x) \le \Lambda ||N|| $$
for any $n \times n$ symmetric matrices $M, N$ with $N \ge 0$, $q \in \mathbb{R}^{n}$, $r \in \mathbb{R}$ and $x \in  \Omega$, and
$\beta$ is a vector-valued function with $||\beta||_{L^{\infty}(\partial \Omega)} = 1 $ and $\beta \cdot \mathbf{n} \ge \delta_{0} $ for some positive constant $\delta_{0}$,
where $ \mathbf{n}$ is the inner unit normal vector field of $\partial \Omega$.

The main purpose of this paper is to find an optimal $W^{2,p}$-regularity theory for \eqref{eloblobs}.
More precisely, we want to identify the minimal condition of $F$ with respect to $x$-variable
under which the Hessian of a solution is as integrable as both the nonhomogeneous term $f$ and the Hessian of the obstacle $\psi$ in the setting of $L^{p}$ spaces for $n<p < \infty$.%, with $n$ to be described later in Section 3.

Throughout this paper, we assume that $F=F(X,q,r,x)$ is convex in $X$ 
and satisfies
\begin{align} \label{scpro}
d(r_{2}-r_{1}) \le F(X,q,r_{1},x)-F(X,q,r_{2},x)
\end{align}
for any $X \in S(n)$, $q \in \mathbb{R}^{n}$, $r_{1}, r_{2} \in \mathbb{R}$ with $r_{1} \le r_{2}$, $x \in \Omega$, and
some $d>0$.
We further assume that
\begin{align} \label{ob_sc} \begin{split}
& \mathcal{M}^{-} (\lambda, \Lambda, X_{1}-X_{2}) -b|q_{1}-q_{2}| - c|r_{1}-r_{2}| \\ & 
\qquad \le F(X_{1},q_{1},r_{1},x) - F(X_{2},q_{2},r_{2},x) \\ &
\qquad \qquad \le \mathcal{M}^{+} (\lambda, \Lambda, X_{1}-X_{2}) +b|q_{1}-q_{2}| + c|r_{1}-r_{2}|
\end{split}
\end{align}
for $X_{1},X_{2} \in S(n) $, $ q_{1}, q_{2} \in \mathbb{R}^{n}$, $r_{1}, r_{2} \in \mathbb{R}$ and $ x \in \Omega $.
These assumptions are essential in order to derive our desired results for solutions of \eqref{eloblobs},
such as the existence, uniqueness and $W^{2,p}$-regularity.

With an oscillation function of $F$ defined as
\begin{align*}  \Theta_{F}(x_{1},x_{2}) := \sup_{X \in S(n) \backslash \{0\}} \frac{|F(X,0,0,x_{1})-F(X,0,0,x_{2})|}{||X||}
\end{align*}
alongside a small perturbation of $\Theta_{F}$ from its integral average in the $L^{n}$-sense, 
we shall prove the $W^{2,p}$-regularity for \eqref{eloblobs}, as we now state the main result of the paper.
We remark that this approach to derive $W^{2,p}$-regularity was employed in \cite{MR1005611}.
\begin{theorem} \label{obob_mthm}
Let $n<p<\infty$.
Assume that $F=F(X,q,r,x)$ is convex in $X$, satisfies \eqref{scpro}-\eqref{ob_sc} and $F(0,0,0,x) \equiv 0$,
 $ \partial \Omega \in C^{3},f \in L^{p}(\Omega), \beta \in C^{2}(\partial \Omega)$ with $\beta \cdot \mathbf{n} \ge \delta_{0}$ for some $\delta_{0} > 0$
and $\psi \in W^{2,p}(\Omega)$.
Then there exists a small $\epsilon=\epsilon(  n, \lambda, \Lambda, p, \delta_{0}, ||\beta||_{C^{2}(\partial \Omega)})>0$
such that if
\begin{align} \label{smallbmo} \sup_{x_{0}\in \overline{\Omega}, 0<\rho<\rho_{0} } \bigg(  \kint_{B_{\rho}(x_{0}) \cap  \Omega }     \Theta_{F}(  x, x_{0})^{n} \ dx  \bigg)^{1/n} \le \epsilon
\end{align}
for some $\rho_{0}>0$,
then there is a unique viscosity solution $u \in W^{2,p}(\Omega)$  of \eqref{eloblobs} with the following estimate
$$||u||_{W^{2,p}(\Omega)} \le c(||f||_{L^{p}(\Omega)}+||\psi||_{W^{2,p}(\Omega)} ) $$
for some constant $c=c(n, \lambda, \Lambda, p, \delta_{0},b,c, ||\beta||_{C^{2}(\partial \Omega)},
\partial \Omega, \diam(\Omega),\rho_{0})$.
\end{theorem}

One of the important issues regarding the obstacle problem is to study solutions near the boundary of the contact surface with the obstacle. 
To this end, suitable approximation methods have been used. 
In this regard, we revisit the argument made in \cite{MR3856840} where the Dirichlet boundary problem was studied instead. 
Our main difficulty in using such an argument comes from the situation that we are treating here the oblique boundary condition.
%First, oblique boundary condition does not give the exact value of a solution on the boundary, unlike in Dirichlet problems. 
%Besides, there is also a difference that our equation has to contain lower-order terms because of the existence issue for the oblique derivative problem.
Thus, we need to modify the tools used in \cite{MR3856840} properly to derive the desired boundary estimates in the present paper.
To do this, we verify several uniform properties of a solution for the corresponding non-obstacle problem such as $W^{2,p}$-regularity and comparison principle.

As a generalization of Neumann boundary problems,
researches on oblique derivative problems have been extensively made as
in \cite{MR738579, MR833695,MR923448, MR1335754, MR1426919, MR1764709}. 
In particular, several notable results for fully nonlinear elliptic equations were obtained in the notion of viscosity solutions. 
The existence and uniqueness of fully nonlinear
oblique derivative problems were proved in \cite{MR1082287,MR1096165, MR1104812}. 
For the regularity of the associated limiting problem, 
there have been established $C^{2,\alpha}$-estimates, see \cite{MR2254613} for the Neumann boundary
condition and \cite{MR3780142} for the oblique boundary condition, respectively. 
In \cite{MR4046185} a global $W^{2,p}$-regularity for the elliptic oblique derivative problem was proved. 

On the other hand, the obstacle problem has been studied along with the free boundary problem. 
We refer the reader to \cite{MR679313,MR2962060} for a general theory of the obstacle problem. 
Regularity results for the elliptic obstacle problem can be found in \cite{MR3198649, MR3353643, MR3542613, MR3906282}. 
We would like to point out that $C^{1}$-regularity for the obstacle problem
with the oblique boundary condition was shown in \cite{MR1875900},
while $W^{2,p}$-regularity for the Dirichlet obstacle problem was given in \cite{MR3856840}.
The main purpose of this paper is to derive a $W^{2,p}$-estimate for the oblique derivative problem
with a $W^{2,p}$-obstacle. 
%As far as we are concerned in the literature our result is the first one
%regarding $W^{2,p}$-regularity for obstacle problems with oblique boundary conditions.
%It is noteworthy to point out that $W^{2, \infty}$-regularity was proved
%in \cite{MR1875900} for a similar problem as in this paper.

The remaining part of the paper is organized in the following way.
In the next section we introduce basic notation and give a brief exposition of viscosity solutions. 
Section 3 deals with the associated oblique derivative problem without obstacles. 
In particular we discuss the existence, uniqueness and $W^{2,p}$-regularity for the non-obstacle problem. 
In the last section we finally give the proof of Theorem \ref{obob_mthm}.

\section{Preliminaries}

\subsection{Notations}
We first introduce some notations which will be used in this paper.
\begin{itemize}
\item $B_{r}(x_{0}):=\{ x \in  \mathbb{R}^{n} :|x-x_{0}|<r \}  $ for $x_{0} \in  \mathbb{R}^{n} $, $r>0$.
$B_{r}=B_{r}(0)$.
\item $S(n)$ is the set of $n \times n $ symmetric matrices and $||M||=\sup_{|x| \le 1} |Mx|$ for any $ M \in S(n)$.
\item We denote the gradient and Hessian of $u$ by $ Du=(D_{1}u, \cdots , D_{n}u)$ and $D^{2}u = (D_{ij}u) $, respectively. Here $ D_{i}u=\frac{\partial u}{\partial x_{i}} $ and $ D_{ij}u=\frac{\partial^{2} u}{\partial x_{i}\partial x_{j}} $ for $ 1 \le i,j \le n$.
\item For any measurable set $A$ with $|A| \neq 0 $ and measurable function $f$, 
to mean the integral average of $f$ over $A$, 
$$ \kint_{A} f \ dx = \frac{1}{|A|} \int_{A} f \ dx. $$
\end{itemize}

\subsection{Basic concepts}
In this subsection, we briefly present some background knowledge for our discussion.
As usual, we are treating a viscosity solution.
To do this, we consider the following problem with oblique boundary data
\begin{align}  \label{elobori}
\left\{ \begin{array}{ll}
F(D^{2}u,Du,u,x) = f & \textrm{in $\Omega, $}\\
\beta \cdot Du=0 & \textrm{on $\partial \Omega$,}\\
\end{array} \right. 
\end{align}
where $\Omega \subset \mathbb{R}^{n}$ is a bounded domain.
There are several ways to define a viscosity solution depending on the choice of a test function.
In this paper, we take a test function $\varphi$ in $W^{2,p}(\Omega)$.
The solution defined in this way is called an $L^{p}$-viscosity solution.
\begin{definition} Let $F$ be continuous in $X$ and measurable in $x$.
Suppose $q > n$ and $f \in L^{q}(\Omega )$. 
A continuous function $u $ is called an $L^{q}$-viscosity solution for \eqref{elobori} if the following conditions hold:
\begin{itemize}
\item[(a)]{(subsolution)} For each $ \varphi \in W^{2,q}(\Omega ) $, whenever $\epsilon>0$, $\mathcal{O}$ is relatively open in $\overline{\Omega}$ and
 $$F( D^{2} \varphi (x ), D \varphi (x ),\varphi (x ), x )  \le f(x)- \epsilon \quad \textrm{a.e. in} \ \mathcal{O} $$ and
 $$ \beta \cdot D\varphi(x) \le - \epsilon \quad \textrm{a.e. on} \ \mathcal{O} \cap \partial \Omega,$$ 
 $u-\varphi$ cannot attain a local maximum in $\mathcal{O}$.
 \\
\item[(b)]{(supersolution)} For each $ \varphi \in W^{2,q}(\Omega) $, whenever $\epsilon>0$, $\mathcal{O}$ is relatively open in $\overline{\Omega}$ and
$$F( D^{2} \varphi (x ), D \varphi (x ),\varphi (x ), x )  \ge f(x)+  \epsilon \quad \textrm{a.e. in} \ \mathcal{O} $$ and 
$$ \beta \cdot D\varphi(x) \ge  \epsilon \quad \textrm{a.e. on} \ \mathcal{O} \cap \partial \Omega,$$ 
$u-\varphi$ cannot attain a local minimum in $\mathcal{O}$. 
\end{itemize}
\end{definition}
We remark that it is also possible to take a $C^{2}$-function as a test function if $F$ is continuous in each variable. 
In this case, the solution is called a $C$-viscosity solution.
For a further discussion of a $C$-viscosity solution, we refer the reader to \cite{MR1351007}.

Next we give some tools to treat viscosity solutions.
To do this, we introduce Pucci extremal operators.
\begin{definition}
Let  $0 < \lambda \le \Lambda $. For any $M \in S(n) $, the Pucci extremal operator $ \mathcal{M}^{+} $ and  $\mathcal{M}^{-}$ are defined as follows:
$$  \mathcal{M}^{+}(\lambda, \Lambda, X)=\Lambda \sum_{e_{i}>0} e_{i} + \lambda \sum_{e_{i}<0} e_{i} $$ and $$  \mathcal{M}^{-}(\lambda, \Lambda, X)=\lambda \sum_{e_{i}>0} e_{i} + \Lambda \sum_{e_{i}<0} e_{i}, $$
where $ e_{i}$ are eigenvalues of $X$. 
Moreover, for $b>0$, we write 
$$ L_{b}^{\pm}(\lambda, \Lambda, u ) = \mathcal{M}^{\pm}(\lambda, \Lambda, D^{2}u) \pm b|Du|, $$
respectively.
\end{definition}

This definition allows us to introduce the class $S$.
These classes can be considered as classes of viscosity solutions.
\begin{definition} \label{classs}
Let  $0 < \lambda \le \Lambda $.
We define the class $\underline{S}(\lambda, \Lambda, b,f) $ \big($ \overline{S}(\lambda, \Lambda,b, f) $, respectively\big) consisting of all functions $u$ such that $$L_{b}^{+}(\lambda,\Lambda,u) \ge f \qquad \big( L_{b}^{-}(\lambda,\Lambda,u) \le f, \textrm{ respectively} \big)$$ in the viscosity sense in $\Omega$. 
We also define
$$S(\lambda, \Lambda,b, f) =\overline{S}(\lambda, \Lambda,b,  f) \cap \underline{S} (\lambda, \Lambda, b, f)  $$ 
and
$$S^{\ast}(\lambda, \Lambda,b,f) =\overline{S}(\lambda, \Lambda,b,| f|) \cap \underline{S} (\lambda, \Lambda, b, -|f|) . $$ 
\end{definition}

\begin{remark}
Let $u$ be a viscosity subsolution (supersolution, respectively) of
$$ F(D^{2}u,Du,u,x)= f \qquad \textrm{in} \ \Omega, $$
where $F=F(X,q,r,x)$ is uniformly elliptic with constants $\lambda, \Lambda$ satisfying the structure condition \eqref{ob_sc}.
Then we can observe that $u$ satisfies
$$L_{b}^{+}(\lambda,\Lambda,u) +F(0,0,u,x) \ge f \qquad \textrm{in} \ \Omega$$
$$ \big(L_{b}^{-}(\lambda,\Lambda,u) +F(0,0,u,x) \le f \qquad \textrm{in} \ \Omega \textrm{, respectively} \big) $$
in the viscosity sense.
\end{remark}

\section{Oblique derivative problems}

Before establishing $W^{2,p}$-regularity ($p>n$) for the obstacle problem \eqref{eloblobs},
we first discuss some issues concerning the existence, uniqueness and regularity for the oblique derivative problem \eqref{elobori}.

%$W^{2,p}$-regularity theory for the equation
%$$ F(D^{2}u,x) = f \qquad \textrm{in} \ \Omega$$
%has been mainly considered in the case of $n<p<\infty$.
%But in many cases, by using the result of \cite{MR1237053}, these results can be extended to the range of $n<p<\infty$ where $n=n-\tau_{0}$ for some small $\tau_{0}=\tau_{0}(n,\lambda,\Lambda, \diam{\Omega} )$. 
%We fix the notation $n$ from now on. 

%We present several types of Alexandroff-Bakelman-Pucci (ABP) estimates
%which are well-known as follows from \cite[Proposition 3.3]{MR1376656}.
%\begin{lemma} \label{abpel}
%Assume that $p>n$ and $\Omega \subset B_{1}$. 
%Let $u\in \overline{S}(\lambda, \Lambda,b, f)$ in $\Omega$ in the viscosity sense.
%Then 
%$$ \sup_{\Omega} u^{-} \le \sup_{\partial \Omega } u^{-}+ C||f^{+}||_{L^{p}(\Omega)},$$
%where $C$ only depends on $n, \lambda,\Lambda$ and $b$.
%\end{lemma}
We recall Definition \ref{classs} to start with an Alexandroff-Bakelman-Pucci (ABP) maximum principle for the oblique boundary problem.
See \cite[Theorem 2.1]{MR3780142} for the proof.
\begin{lemma} \label{obabpsp}
Let $u$ satisfy
\begin{align*} 
\left\{ \begin{array}{ll}
u \in \overline{S}(\lambda, \Lambda, b,f) & \textrm{in $\Omega, $}\\
\beta \cdot Du \le g & \textrm{on $\Gamma \subset \partial \Omega$ }\\ 
\end{array} \right. 
\end{align*}
in the viscosity sense for $f \in L^{n}(\Omega)$ and $g \in L^{\infty}(\Gamma)$.
Suppose that there exist $\xi \in \partial B_{1}$ and $\delta_{1} >0$ such that $\beta (x) \cdot \xi \ge \delta_{1}$ for any $ x \in \Gamma$.
Then 
$$ \sup_{\Omega} u^{-} \le \sup_{\partial \Omega \backslash \Gamma} u^{-}+C(||f^{+}||_{L^{n}(\Omega)}+\max_{\Gamma}g^{+}),$$
where $C$ only depends on $n, \lambda, \Lambda,b,\delta_{1}$ and $\diam (\Omega)$.
\end{lemma}

We also give a weak Harnack's inequality for supersolutions.
%Since $W^{2,p}$-strong solutions are $W^{2,p}$-viscosity supersolutions, we can also use this result
%for our solution. 
The proof can be found in \cite[Proposition 1.8]{MR2486925}. 
\begin{lemma} \label{weha}
Let $p > n$ and $f \in L^{p}(B_{1})$.
Suppose that $u \in \overline{S}(\lambda, \Lambda,b, f)$ in the viscosity sense and $u \ge 0$ in $B_{1}$.
Then there exist $p_{0}, C > 0$ depending only on $n, \lambda,\Lambda$ and $b$ such that
$$||u||_{L^{p_{0}}(B_{1/2})} \le C(\inf_{B_{1/2}}u +  ||f||_{L^{p}(B_{1})}).$$
\end{lemma}

The following stability lemma, which can be found in \cite[Proposition 1.5]{MR2486925},
will be used later
(see also \cite[Theorem 3.8]{MR1376656}).

\begin{proposition} \label{win15} 
For $k \in \mathbb{N}$, let $ \Omega_{k} \subset \Omega_{k+1} $ be an increasing sequence of bounded domains and $\Omega : = \cup_{k \ge 1} \Omega_{k} $.
Let $F$ and $F_{k}$ be measurable in $x$ and satisfy the structure condition \eqref{ob_sc}.
Assume that for $p>n$, $f \in L^{p}(\Omega) $ and $ f_{k} \in L^{p}(\Omega_{k})$, and that $ u_{k } \in C(\Omega_{k})$ are $L^{p}$-viscosity subsolutions (supersolutions, respectively) of $F_{k}(D^{2}u_{k}, Du_{k}, u_{k}, x) = f_{k}$ in $\Omega_{k} $.
Suppose that $u_{k} \to u$ locally uniformly in $\Omega$, and for $B_{r}(x_{0}) \subset \Omega$ and $ \varphi \in W^{2,p}(B_{r}(x_{0}))$
\begin{align} \label{obcc} ||(s-s_{k})^{+} ||_{L^{p}(B_{r}(x_{0}))} \to 0 \qquad  \big( ||(s-s_{k})^{-} ||_{L^{p}(B_{r}(x_{0}))} \to 0   \big),
\end{align}
where $ s(x) = F(D^{2}\varphi, D \varphi, u, x)-f(x)$ and $ s_{k}(x) = F(D^{2}\varphi_{k}, D \varphi_{k}, u_{k}, x)-f_{k}(x)$.
Then $u$ is an $L^{p}$-viscosity subsolution (supersolution) of $$  F(D^{2}u, Du, u, x)=f(x) \ \textrm{ in} \  \Omega.$$
Moreover, if $ F$ and $f$ are continuous, then $u$ is a $C$-viscosity subsolution (supersolution) provided that  \eqref{obcc} holds for $ \varphi \in C^{2}(B_{r}(x_{0}) )$.
\end{proposition}

Now we return to \eqref{elobori}.
Here we assume that there is a continuous increasing function $\omega$,
defined on $[0, \infty)$ with $\omega(0)=0$, such that
\begin{align} \label{sc_add}
F(X_{1},q,r,x_{1})-F(X_{2},q,r,x_{2}) \le \omega (|x_{1}-x_{2}|(|q|+1)+\alpha|x_{1}-x_{2}|^{2})
\end{align}
holds for any $x_{1},x_{2} \in \Omega$, $q \in \mathbb{R}^{n}$, $r \in \mathbb{R}$, $\alpha >0$ and $X_{1},X_{2} \in S(n)$ satisfying
\begin{align} \label{sc_add2}
-3\alpha \left(\begin{array}{c c}
I & 0 \\
0 & I
\end{array}\right) 
\le 
\left(\begin{array}{c c}
X_{2} & 0 \\
0 & -X_{1}
\end{array}\right) 
\le 3\alpha
\left(\begin{array}{c c}
I & -I \\
-I & I
\end{array}\right) .
\end{align}
%We remark that this assumption is needed to guarantee the existence and uniqueness for solutions of oblique derivative problems, Lemma \ref{ob_exun}.

One can find the following existence and uniqueness for the problem \eqref{elobori} in \cite[Theorem 7.19]{MR3059278}.
We remark that the condition \eqref{sc_add} is needed to ensure this lemma. 
\begin{lemma} \label{ob_exun}
Assume that $F=F(X,q,r,x)$ is convex in $X$ and continuous in $x$, satisfies \eqref{scpro}-\eqref{ob_sc} and \eqref{sc_add}, and $F(0,0,0,x) \equiv 0$,  $ \partial \Omega \in C^{3}$, $f \in L^{p}(\Omega)\cap C(\overline{\Omega})$ for $p>n$, $\beta \in C^{2}(\partial \Omega)$ with $\beta \cdot \mathbf{n} \ge \delta_{0}$ for some $\delta_{0} > 0$. 
Then there exists a unique viscosity solution $u$ of \eqref{elobori}.
\end{lemma}

We now have the following $W^{2,p}$-estimate for the viscosity solution to \eqref{elobori}.

\begin{lemma}\label{ob_mthmed}
Under the assumptions and conclusion in Lemma \ref{ob_exun},
there exists a small $\epsilon=\epsilon(  n, \lambda, \Lambda, p, \delta_{0}  )$
such that if
\begin{align} \label{smallbmo} \sup_{x_{0}\in \overline{\Omega}, 0<\rho<\rho_{0} } \bigg(  \kint_{B_{\rho}(x_{0}) \cap  \Omega }     \Theta_{F}(  x, x_{0})^{n} \ dx  \bigg)^{1/n} \le \epsilon 
\end{align}
for some $\rho_{0}>0$,
then the unique solution $u$ belongs to $W^{2,p}(\Omega)$ with the following estimate
\begin{align} \label{estwint}
||u||_{ W^{2,p}(\Omega)} \le C||f||_{ L^{p}( \Omega )}
\end{align}
for some $C=C(n, \lambda, \Lambda, p, \delta_{0},b,c, ||\beta||_{C^{2}(\partial \Omega)}, \diam(\Omega),\rho_{0})$.
\end{lemma}

\begin{proof} 
According to \cite[Theorem 4.6]{MR4046185}, $u$ belongs to $W^{2,p}(\Omega)$ with
\begin{align} \label{estbh20} ||u||_{ W^{2,p}(\Omega)} \le C (||u||_{ L^{\infty}(\Omega )}+||f||_{ L^{p}( \Omega )}) 
\end{align}
for some $C=C(n, \lambda, \Lambda, p, \delta_{0},b,c, ||\beta||_{C^{2}(\partial \Omega)}, \diam(\Omega),\rho_{0})$.

Therefore, it suffices to obtain the estimate \eqref{estwint}.
To prove this, we argue by contradiction.
Suppose not.
Then there exist sequences $\{ u_{k} \}$ and $\{ f_{k} \}$ such that
$u_{k}$ is the viscosity solution of
\begin{align}
\left\{ \begin{array}{ll}
F(D^{2}u_{k},Du_{k},u_{k},x) = f_{k} & \textrm{in $\Omega, $}\\
\beta \cdot Du_{k}=0 & \textrm{on $\partial \Omega$}\\
\end{array} \right. 
\end{align}
with
\begin{align} \label{divseq} ||u_{k}||_{W^{2,p}(\Omega)} > k ||f_{k}||_{L^{p}(\Omega)} \qquad \textrm{for each }k \ge 1.
\end{align}

Consider $\tilde{u}_{k}=t_{k}^{-1}u_{k}$, $\tilde{f}_{k}=t_{k}^{-1}f_{k}$ and
$ \tilde{F}_{k}(X,q,r,x)=t_{k}^{-1}F(t_{k}X,t_{k}q,t_{k}r,x),$
where $t_{k}=||u_{k}||_{W^{2,p}(\Omega)}$.
Then $\tilde{u}_{k}$ is a viscosity solution of
\begin{align*}
\left\{ \begin{array}{ll}
\tilde{F}_{k}(D^{2}\tilde{u}_{k},D\tilde{u}_{k},\tilde{u}_{k},x) = \tilde{f}_{k} & \textrm{in $\Omega, $}\\
\beta \cdot D\tilde{u}_{k}=0 & \textrm{on $\partial \Omega$.}\\
\end{array} \right. 
\end{align*}
We also see that $\tilde{F}_{k}$ satisfies the assumptions of Lemma \ref{ob_exun} and $ ||\tilde{u}_{k}||_{W^{2,p}(\Omega)} = 1$. 
By \eqref{divseq},
$||\tilde{f}_{k}||_{L^{p}(\Omega)} < 1/k $ and this tends to zero as $k \to \infty$. 
Moreover, by weak compactness theorem, we can extract a proper subsequence $\{ \tilde{u}_{k_{j}} \} \subset \{ \tilde{u}_{k} \}$
such that 
\begin{align*}
\left\{ \begin{array}{ll}
\tilde{u}_{k_{j}} \rightharpoonup \tilde{v}  & \textrm{in $W^{2,p}(\Omega), $}\\
\tilde{u}_{k_{j}}\to \tilde{v} & \textrm{in $W^{1,p}(\Omega)$}\\
\end{array} \right. 
\end{align*}
for some $\tilde{v} \in W^{2,p}(\Omega)$.
Since $p>n$, we also observe that $W^{1,p}(\Omega) \subset \subset C(\overline{\Omega})$
and this yields $\tilde{u}_{k_{j}} \to \tilde{v}$ in $C(\overline{\Omega})$.
Moreover, we also observe that for each $j$, $ \tilde{u}_{k_{j}}  \in C^{0,\alpha_{0}}(\partial \Omega)$ for some $0<\alpha_{0}<1-n/p$ and $|| D\tilde{u}_{k_{j}} ||_{L^{\infty}(\partial \Omega)} \le C$ for some $C=C(n,p,\Omega)>0$.
Then, from Arzel\'{a}-Ascoli criterion, we get
$$ D\tilde{u}_{k_{j}} \to D\tilde{v} \qquad \textrm{on} \ \partial \Omega.$$
Hence, by using Proposition \ref{win15}, we see that $v$ is a viscosity solution of
\begin{align} \label{oblemlim} 
\left\{ \begin{array}{ll}
\tilde{F}(D^{2}\tilde{v},D\tilde{v},\tilde{v},x) = 0 & \textrm{in $\Omega, $}\\
\beta \cdot D\tilde{v}=0 & \textrm{on $\partial \Omega$}\\
\end{array} \right. 
\end{align}
for some $\tilde{F}=\tilde{F}(X,q,r,x)$.
Here we can check that  $\tilde{F}$ also satisfies \eqref{scpro}-\eqref{ob_sc} and \eqref{sc_add}.

Now we deduce that $\tilde{v} \equiv 0$ solves \eqref{oblemlim}.
Then $\tilde{v}$ is the unique solution of \eqref{oblemlim} by Lemma \ref{ob_exun}.
But, in this case, we get
\begin{align*}
1 = ||\tilde{u}_{k_{j}}||_{W^{2,p}(\Omega)}  \le C(||\tilde{u}_{k_{j}}||_{ L^{\infty}(\Omega )}+||\tilde{f}_{k_{j}}||_{ L^{p}( \Omega )})  \to 0 \qquad \textrm{as} \ j \to \infty ,
\end{align*}
which is a contradiction.
\end{proof}

In the above lemma, we have assumed \eqref{sc_add}, which says that $F$ and $f$ are at least continuous in $x$.
By using mollification, we can relax this assumption.
\begin{lemma} \label{ob_mthmgen}
Assume that $F=F(X,q,r,x)$ is convex in $X$ and measurable in $x$, satisfies \eqref{scpro}-\eqref{ob_sc} and $F(0,0,0,x) \equiv 0$,  $ \partial \Omega \in C^{3}$, $f \in L^{p}(\Omega)$ for $p>n$, $\beta \in C^{2}(\partial \Omega)$ with $\beta \cdot \mathbf{n} \ge \delta_{0}$ for some $\delta_{0} > 0$. 
Then there exists a unique viscosity solution $u$ of \eqref{elobori} with the estimate \eqref{estwint}.
\end{lemma}
\begin{proof}
Fix $\epsilon>0$.
With a standard mollifier $\varphi$ having $\supp \varphi \subset B_{1}$,
we define $\varphi_{\epsilon} (x)=\epsilon^{-n} \varphi (x / \epsilon )$.
Then we set $f^{\epsilon}(x)=(f \ast \varphi_{\epsilon})(x)$ and
$$ F^{\epsilon}(X,q,r,x)= (F(X,q,r,\cdot)\ast \varphi_{\epsilon})(x).$$
Note that we extended $F$ and $f$ to zero outside $\Omega$ here. 
Then one can check that $ f^{\epsilon}\in L^{p}(\mathbb{R}^{n}) \cap C^{\infty}(\mathbb{R}^{n})$ and
$F^{\epsilon}$ is convex in $X$ and $F^{\epsilon}(0,0,0,x) \equiv 0$.
Furthermore, we also observe that $F^{\epsilon}$ satisfies \eqref{scpro}-\eqref{ob_sc}, \eqref{sc_add} and \eqref{smallbmo}
(see the proof of \cite[Theorem 4.3]{MR2486925}).
Consider the following problem
\begin{align} \label{applwoo}
\left\{ \begin{array}{ll}
F^{\epsilon}(D^{2}u_{\epsilon},Du_{\epsilon},u_{\epsilon},x) = f^{\epsilon} & \textrm{in $\Omega, $}\\
\beta \cdot Du_{\epsilon}=0 & \textrm{on $\partial \Omega$.}\\
\end{array} \right. 
\end{align}
Then applying Lemma \eqref{ob_mthmed} to $u_{\epsilon}$, there exists the unique solution $u_{\epsilon}$ of \eqref{applwoo} with
\begin{align*} 
||u_{\epsilon}||_{ W^{2,p}(\Omega)} \le C||f^{\epsilon}||_{ L^{p}( \Omega )}
\end{align*}
for some $C=C(n, \lambda, \Lambda, p, \delta_{0},b,c, ||\beta||_{C^{2}(\partial \Omega)}, \diam(\Omega),\rho_{0})$.

Since $W^{2,p}(\Omega) \subset \subset C^{ 1, \alpha} (\overline{\Omega})$ with $0< \alpha < 1-n/p$ by Sobolev imbedding,
we have 
$\{ u_{\epsilon} \}_{\epsilon >0} $ is uniformly bounded in $C^{1,\alpha_{0}} (\overline{\Omega})$ for any small $\epsilon>0$ and some $0<\alpha_{0}<1-n/p$.
Thus, by using Arzel\'{a}-Ascoli criterion, we can obtain that there exists a function $v$ with
$$ u_{\epsilon_{j}} \to v \qquad \textrm{uniformly in}\ \overline{\Omega} $$
for some subsequence $\{ u_{\epsilon_{j}} \} \subset \{ u_{\epsilon} \}$.  
Again, applying Proposition \ref{win15} to $v$, we can derive that
$v$ is a viscosity solution of \eqref{elobori}
with
\begin{align*} 
||v||_{ W^{2,p}(\Omega)} \le C||f||_{ L^{p}( \Omega )}
\end{align*}
for some $C=C(n, \lambda, \Lambda, p, \delta_{0},b,c, ||\beta||_{C^{2}(\partial \Omega)}, \diam(\Omega),\rho_{0})$.
The uniqueness can be deduced by Lemma \ref{obcom} below.
\end{proof}

%On the other hand, we observe that
%$$ \beta \cdot D\psi^{\epsilon_{j}} = \beta \cdot (D\psi \ast \varphi_{\epsilon_{j}}).$$
%Since $\psi \in W^{2,p}(\Omega)$ with $p>n$, it can be seen that $D\psi  \in C( \overline{\Omega})$.
%Thus, we can observe that $D\psi^{\epsilon_{j}} $ converges uniformly to $D\psi $ in $B_{\delta}(x)$ for some sufficiently small $\delta$. 
%Hence,
%\begin{align*}
%\beta (x) \cdot  (D\psi \ast \varphi_{\epsilon_{j}})(x) &
%= \beta (x) \cdot \int_{B_{\epsilon_{j}}(x)} D \psi (y) \varphi_{\epsilon_{j}}(x-y) dy \\ &
%= \int_{B_{\epsilon_{j}}(x)} ( \beta (x) \cdot  D \psi (y) )\varphi_{\epsilon_{j}}(x-y) dy
%\end{align*}
%for every $x \in \partial \Omega$.
%Combining the assumption $\beta \cdot D \psi \ge \nu >0$ with the uniform convergence of $\psi^{\epsilon_{j}}$,
%we obtain $\beta (x) \cdot  D \psi (y) \ge 0$ in $B_{\epsilon_{j}}(x) $ for large $j$.
%Therefore, we also obtain $\beta \cdot D\psi^{\epsilon_{j}} \ge 0$ for any sufficiently large $j$.

Meanwhile, a comparison principle for \eqref{elobori} can be also obtained as in the case of Dirichlet problems,
see \cite[Theorem 2.10]{MR1376656}. 
This will be used to prove Theorem \ref{obob_mthm} in the next section. 
\begin{lemma}\label{obcom} 
Let $\Omega_{0} \subset \Omega$, and let $\Gamma \in C^{2}$ be relatively open in $\partial \Omega$.
Suppose that  $F=F(X,q,r,x)$ is convex in $X$ and continuous in $x$, satisfies \eqref{scpro}-\eqref{sc_add} and $F(0,0,0,x) \equiv 0$, $\beta \in C^{2}(\overline{\Gamma}) $ with $\beta \cdot \mathbf{n} \ge \delta_{0}$ for some $\delta_{0}>0$, $\psi \in C(\partial \Omega_{0} \backslash \Gamma)$ and $f \in L^{p}(\Omega_{0})$ for $n< p < \infty$.

Let $u_{1}, u_{2} \in W^{2,p}(\Omega) \cap C(\overline{\Omega})$ satisfy
\begin{align*}
\left\{ \begin{array}{ll}
F(D^{2}u_{1},Du_{1}, u_{1},x) \le f & \textrm{in $\Omega_{0}, $}\\
u_{1}\ge \psi & \textrm{on $\partial \Omega_{0} \backslash \Gamma$,} \\
\beta \cdot Du_{1} \le 0 & \textrm{on $ \Gamma$}\\
\end{array} \right. 
\end{align*}
and
\begin{align*}
\left\{ \begin{array}{ll}
F(D^{2}u_{2},Du_{2},u_{2},x) \ge f & \textrm{in $\Omega_{0}, $}\\
u_{2}\le \psi & \textrm{on $\partial \Omega_{0} \backslash \Gamma$,} \\
\beta \cdot Du_{2} \ge 0 & \textrm{on $ \Gamma$ }\\
\end{array} \right. 
\end{align*}
in the viscosity sense.
Then we have $u_{1} \ge u_{2}$ in $\Omega_{0}$. 
\end{lemma}
\begin{proof}
First we set 
$$ G(X,q,r,x)=F(X+D^{2}u_{2},q+Du_{2},r+u_{2},x)-F(D^{2}u_{2},Du_{2},u_{2},x).$$
One can see that $G$ satisfies \eqref{scpro}-\eqref{sc_add}, $w :=u_{1}-u_{2}$ solves
\begin{align*}
G(D^{2}w,Dw,w,x)=F(D^{2}u_{1},Du_{1},u_{1},x)-F(D^{2}u_{2},Du_{2},u_{2},x) \le 0
\end{align*}
in the viscosity sense,
and that
$ w \in \overline{S}(\lambda , \Lambda, b, -g)$ for
$$ g(x)=  G(0,0,w,x)=F(D^{2}u_{2},Du_{2},u_{1},x)-F(D^{2}u_{2},Du_{2},u_{2},x).$$

Set $$V:=\{ x \in \overline{\Omega}_{0} : w(x) <0 \} = \{x \in \overline{\Omega}_{0} :  u_{1}(x) < u_{2}(x) \} .$$
We want to claim that $V = \varnothing$.
Suppose not.
Then we have $\inf_{V}w <0$.
Since
$$g(x) > d(u_{2}-u_{1})(x) = -dw(x) >0 \qquad \textrm{for} \ x \in V , $$
$w$ satisfies
\begin{align*}
\left\{ \begin{array}{ll}
w \in \overline{S}(\lambda, \Lambda, b,0) & \textrm{in $V, $}\\
w\ge0 & \textrm{on $\partial V \backslash \Gamma$,} \\
\beta \cdot Dw \le 0 & \textrm{on $ \Gamma$}\\
\end{array} \right. 
\end{align*}
in the viscosity sense, according to \cite[Theorem 3.1]{MR3780142}. 

We first consider the case $\partial V \backslash \Gamma \neq \varnothing$. 
Observe that $w \equiv 0$ on $\partial V \cap \Omega_{0}$ and
$ w \ge 0$ on $(\partial V \cap \partial \Omega_{0} ) \backslash \Gamma$. 
From ABP maximum principle (see \cite[Proposition 3.3]{MR1376656}), we can deduce that
$$ \sup_{V} w^{-} =  \sup_{\partial V} w^{-}.$$
Thus, $w$ attains a minimum point $x_{0}$ on $\partial V \cap \Gamma$.

Now define $\tilde{w}(x) = w(x) -\inf_{V} w$ for $x \in V$.
Then we see that $\tilde{w} \ge 0 $ in $\overline{V}$ with $\tilde{w}(x_{0})=0$.
Consider a small neighborhood $N(x_{0}) \subset \Omega \cup \Gamma$ of $x_{0}$.
According to \cite[Proposition 11]{MR2813890}, we can deduce that  there exists a point $x_{1} \in N(x_{0})\cap \Omega$ with 
$\tilde{w}(x_{1})=0$ since  $\beta \cdot Dw \le 0$.
Then, by Lemma \ref{weha}, there exists a small number $\rho>0$ with $B_{\rho}(x_{1}) \subset \subset \Omega$ such that
$$ ||\tilde{w}||_{L^{p_{0}}(B_{\rho/2}(x_{1}))} =0 .$$
That is,
$$ \tilde{w} \equiv 0  \qquad \textrm{in}\quad B_{\rho/2}(x_{1}).$$
Repeating the above procedure, we can deduce that 
$$ \tilde{w} \equiv 0  \qquad \textrm{in}\quad V.$$
Since $\tilde{w}$ is continuous in $\overline{V}$, we have
$$ \tilde{w} \equiv 0  \qquad \textrm{in}\quad \overline{V}, $$
and this implies $w= \tilde{w}$.
However, it is a contradiction, as we have assumed that
$ \inf_{V}w <0.$
Hence, $V = \varnothing$ and we conclude that
$$ w \ge 0 \qquad \textrm{in}\quad \Omega_{0},$$
and this implies
$$ u_{1} \ge u_{2} \qquad \textrm{in}\quad \Omega_{0}.$$

On the other hand, if $ \Gamma = \partial V $, we have
\begin{align*}
\left\{ \begin{array}{ll}
w \in \overline{S}(\lambda, \Lambda, b,0) & \textrm{in $V, $}\\
\beta \cdot Dw \le 0 & \textrm{on $ \Gamma$.}\\
\end{array} \right. 
\end{align*}
Again, by ABP maximum principle, we also have 
$$ \sup_{V} w^{-} =  \sup_{\partial V} w^{-}<0.$$
Set $\tilde{w}  = w -\inf_{V} w$ in $V$.
Then there is a point $y_{0} \in \partial V$ with $\tilde{w}(y_{0})=0$.
From Lemma \ref{obabpsp}, we can see that there is an interior point $y_{1}$ such that $\tilde{w}(y_{1})=0$.
Now we derive that $\tilde{w} \equiv 0$ in V by using a similar argument as above.
This yields that $w \equiv c_{1} $ in $V$ for some $c_{1}<0$.

By the definition of $w$, we have $u_{1}\equiv u_{2}+c_{1}$ in $V$.
Then we can observe that
\begin{align*}
F(D^{2}u_{1},Du_{1}, u_{1},x)&=F(D^{2}(u_{2}+c_{1}),D(u_{2}+c_{1}), u_{2}+c_{1},x)\\
& =F(D^{2}u_{2},Du_{2}, u_{2}+c_{1},x) \\
& \ge F(D^{2}u_{2},Du_{2}, u_{2},x) - dc_{1} \\
& > f
\end{align*}
in $V$.
But it is a contradiction because $F(D^{2}u_{1},Du_{1}, u_{1},x) \le f$.
Therefore, we conclude that $V = \varnothing$.
This completes the proof.
\end{proof}

%\begin{remark}
%In \cite{MR3780142} and \cite{MR4046185}, the case of $C^{2}$-viscosity solutions was only considered. 
%But those results can be extended to the case of $W^{2,p}$-viscosity solutions without difficulty.
%\end{remark}
\section{Proof of Theorem \ref{obob_mthm}}

In this section we establish our main result, Theorem \ref{obob_mthm}.
Our strategy is to construct a sequence of approximating oblique derivative problems to \eqref{eloblobs}.
This construction makes it possible to utilize those results obtained in the previous section for \eqref{elobori}.

In the process of the proof, we are going to use the following Schauder's fixed point theorem (see  \cite[Theorem V.9.5]{MR768926}).
\begin{lemma}[Schauder's fixed point theorem] \label{sfpt}
Assume that $X$ is a Banach space, $K \subset X$ is closed, bounded and convex, and suppose that $S:K \to K$ is compact.
Then $S$ has a fixed point in $K$.
\end{lemma}

We now prove the main result of this paper.

\begin{proof}[Proof of Theorem \ref{obob_mthm}]
Fix $\epsilon>0$ and choose a non-decreasing function $\Phi_{\epsilon} \in C^{\infty}(\mathbb{R})$ such that
\begin{align} \label{apfcon1} \Phi_{\epsilon}(s)\equiv0 \quad \textrm{if}  \quad s \le 0 ;  \qquad
 \Phi_{\epsilon}(s)\equiv1 \quad \textrm{if}  \quad s \ge \epsilon ,
 \end{align}
and
\begin{align} \label{apfcon2}  0 \le \Phi_{\epsilon}(s) \le 1 \quad \textrm{for any} \ \, s \in \mathbb{R} .\end{align}

Set $$g(x) := f(x)-F(D^{2}\psi, D\psi, \psi, x).$$
Then we have $g \in L^{p}(\Omega) $ with the estimate 
\begin{align} \label{obsgest} \begin{split}
||g||_{ L^{p}(\Omega)} & \le ||f||_{ L^{p}(\Omega)} + ||F(D^{2} \psi,D\psi, \psi, x)||_{ L^{p}(\Omega)}
\\ & \le C( ||f||_{ L^{p}(\Omega)} + ||\psi||_{ W^{2,p}(\Omega)} ) \end{split}
\end{align}
for some $C=C(n, \lambda, \Lambda,b,c)>0$, since $f, D^{2}\psi \in  L^{p}(\Omega) $ and $F$ satisfies 
\eqref{ob_sc}.

We now consider the following oblique derivative problem without obstacles
\begin{align} \label{oblobsapp}
\left\{ \begin{array}{ll}
F(D^{2}u_{\epsilon},Du_{\epsilon},u_{\epsilon},x) = g^{+} \Phi_{\epsilon}(u_{\epsilon}-\psi)+ f- g^{+} & \textrm{in $\Omega, $}\\
\beta \cdot Du_{\epsilon}=0 & \textrm{on $\partial \Omega$.}\\
\end{array} \right. 
\end{align}
We want to show that \eqref{oblobsapp} has a unique viscosity solution.
For this, fix a function $v_{0} \in L^{p}(\Omega)$.
Then according to Lemma \ref{ob_mthmgen},
we know that there exists a unique viscosity solution $v_{\epsilon} \in W^{2,p}(\Omega)$ of
\begin{align*} 
\left\{ \begin{array}{ll}
F(D^{2}v_{\epsilon},Dv_{\epsilon},v_{\epsilon},x) = g^{+} \Phi_{\epsilon}(v_{0}-\psi)+ f - g^{+} & \textrm{in $\Omega, $}\\
\beta \cdot Dv_{\epsilon}=0 & \textrm{on $\partial \Omega$}
\end{array} \right. 
\end{align*}
with the estimate 
\begin{align*}
||v_{\epsilon}||_{ W^{2,p}(\Omega)} & \le C|| g^{+} \Phi_{\epsilon}(v_{0}-\psi)+ f - g^{+} ||_{ L^{p}(\Omega)}
\\ & \le C(|| f ||_{ L^{p}(\Omega)}+||g ||_{ L^{p}(\Omega)})
\\& \le C(|| f ||_{ L^{p}(\Omega)}+||\psi ||_{ W^{2,p}(\Omega)})
\end{align*}
for some $C=C(n, \lambda, \Lambda, p, \delta_{0},b,c,   ||\beta||_{C^{2}(\partial \Omega)},\rho_{0}, \diam(\Omega))>0$, where we have used \eqref{apfcon1} and \eqref{apfcon2}.
Thus, $$||v_{\epsilon}||_{ W^{2,p}(\Omega)}   \le C_{0} $$ for some 
$$C_{0}=C_{0}(n, \lambda, \Lambda, p, \delta_{0},b,c,||\beta||_{C^{2}(\partial \Omega)} \diam(\Omega), || f||_{ L^{p}(\Omega)},||\psi||_{ L^{p}(\Omega)},\rho_{0}).$$
Note that $C_{0}$ is independent of $v_{0}$.
Now we can define a nonlinear operator  $S_{\epsilon}:L^{p}(\Omega) \to W^{2,p}(\Omega) \subset L^{p}(\Omega)$ such that $S_{\epsilon}v_{0}=v_{\epsilon}$ with \eqref{oblobsapp}.
Write
$$K:= \{ h \in L^{p}(\Omega) : ||h||_{ L^{p}(\Omega) } \le C_{0} \}.$$
Note that $K$ is a closed convex subset of $L^{p}(\Omega)$. 
On the other hand, by Rellich-Kondrachov compactness theorem, we observe that $W^{2,p}(\Omega)$ is compactly imbedded in $W^{1,p}(\Omega)$ and so is in $L^{p}(\Omega)$. 
Hence, the closure of $S_{\epsilon}(A)$ is compact for every $A \subset K$. 
Meanwhile, by Proposition \ref{win15}, we can also conclude that $S_{\epsilon}$ is a continuous operator.
%Now we need to check that $S$ is continuous.
%Consider a sequence $\{ v_{k} \} \subset L^{p}(\Omega)$ which converges to a function $v$ in $L^{p}(\Omega)$
%and let $w_{k}$ and $w$ be solutions of 
%\begin{align*} 
%\left\{ \begin{array}{ll}
%F(D^{2}w_{k},x) = g^{+} \Phi_{\epsilon}(v_{k}-\psi)+ f - g^{+} & \textrm{in $\Omega, $}\\
%\beta \cdot Dw_{k}=0 & \textrm{on $\partial \Omega$,}\\
%\int_{\Omega}w_{k} \ dx = 0 \\
%\end{array} \right. 
%\end{align*}
%and
%\begin{align*} 
%\left\{ \begin{array}{ll}
%F(D^{2}w,x) = g^{+} \Phi_{\epsilon}(v-\psi)+ f - g^{+} & \textrm{in $\Omega, $}\\
%\beta \cdot Dw=0 & \textrm{on $\partial \Omega$,}\\
%\int_{\Omega}w \ dx = 0 ,\\
%\end{array} \right. 
%\end{align*}
%respectively.
%Then we see that 
%\begin{align*} 
%\left\{ \begin{array}{ll}
%w_{k}-w \in S(\lambda, \Lambda, g^{+} \Phi_{\epsilon}(v_{k}-\psi)-g^{+} \Phi_{\epsilon}(v-\psi)) & \textrm{in $\Omega, $}\\
%\beta \cdot D(w_{k}-w)=0 & \textrm{on $\partial \Omega$,}\\
%\int_{\Omega}(w_{k}-w) \ dx = 0 .\\
%\end{array} \right. 
%\end{align*}
%Now we observe that 
%$$|| g^{+} \Phi_{\epsilon}(v_{k}-\psi)-g^{+} \Phi_{\epsilon}(v-\psi) ||_{L^{p}(\Omega)}\to 0 \qquad \textrm{as} \ k \to \infty$$
%since $v_{k} \to v$ in $L^{p}(\Omega)$.
%By using the argument in the proof of Lemma \ref{obsoun}, we obtain
%$$ ||w_{k}-w||_{L^{\infty}(\Omega)}\to 0 \qquad \textrm{as} \quad k \to \infty$$
%and this implies
%$$ ||w_{k}-w||_{L^{p}(\Omega)}\to 0 \qquad \textrm{as} \quad k \to \infty.$$

Therefore, by Lemma \ref{sfpt}, there exists a function $u_{\epsilon} \in K$ satisfying $S_{\epsilon}u_{\epsilon}=u_{\epsilon}$,
and this implies that $u_{\epsilon}$ is a viscosity solution of \eqref{oblobsapp}.
Furthermore, from \eqref{obsgest} and Lemma \ref{ob_mthmgen}, we also observe that
\begin{align} \label{uepest}  
||u_{\epsilon}||_{ W^{2,p}(\Omega)} \le C(|| f ||_{ L^{p}(\Omega)}+||\psi ||_{ W^{2,p}(\Omega)})
\end{align}
for some $C=C(n, \lambda, \Lambda, p, \delta_{0},b,c,   ||\beta||_{C^{2}(\partial \Omega)},\diam(\Omega), \rho_{0})$. 
This shows that $\{ u_{\epsilon} \}_{\epsilon>0} $ is uniformly bounded in $W^{2,p}(\Omega)$.

Recall that $p>n$. Then we observe that $W^{2,p}(\Omega) \subset \subset C^{1, \alpha_{0}}(\overline{\Omega})$ for some $0<\alpha_{0}< 1-n/p $ by Morrey imbedding.
Therefore, we can find a subsequence $\{ u_{\epsilon_{j}} \}$ with $\epsilon_{j} \searrow 0$ and a function $u \in  W^{2,p}(\Omega) $ such that
\begin{align}\label{uepcon} 
\left\{ \begin{array}{ll}
u_{\epsilon_{j}} \rightharpoonup u  \quad \textrm{in} \ W^{2,p}(\Omega),\\
u_{\epsilon_{j}} \to u  \quad \textrm{in} \ W^{1,p}(\Omega) \subset C^{0,\alpha_{0}}(\overline{\Omega})\\
\end{array} \right. 
\end{align}
as $j \to \infty$.

Now we claim that $u$ is indeed the unique viscosity solution of \eqref{eloblobs}.
We first see that $u$ is uniformly bounded and equicontiuous on $\partial \Omega$ from \eqref{uepest} and Morrey imbedding.
Thus, by using Arzel\'{a}-Ascoli criterion, we have
$$\beta \cdot Du=0 \quad \textrm{on} \ \partial \Omega$$ in the viscosity sense.
On the other hand, from \eqref{oblobsapp}, we observe that
$$F(D^{2}u_{\epsilon_{j}},Du_{\epsilon_{j}},u_{\epsilon_{j}},x) = g^{+} \Phi_{\epsilon_{j}}(u_{\epsilon_{j}}-\psi)+ f- g^{+} \le f \qquad \textrm{in} \ \Omega  $$
for each $j$.
Passing to the limit $j \to \infty$, we have $F(D^{2}u,Du,u,x) \le f$ in $\Omega$.

Next, we show that
$$ u \ge \psi \quad \textrm{in} \  \overline{\Omega}.$$
We first see that $ \Phi_{\epsilon_{j}}(u_{\epsilon_{j}}-\psi) \equiv 0$ 
on the set $$V_{j}=\{ x \in \overline{\Omega} :u_{\epsilon_{j}}(x) < \psi(x) \}.$$
If $V_{j}=\varnothing$, we have $u_{\epsilon_{j}} \ge \psi$ in $\overline{\Omega}$, and so we are done.
Now suppose that $V_{j} \neq \varnothing$. Then,
$$ F(D^{2}u_{\epsilon_{j}},Du_{\epsilon_{j}}, u_{\epsilon_{j}},x) = f(x)-g^{+}(x) \qquad \textrm{for} \ x \in V_{j}. $$  
We note that $V_{j}$ is relatively open in $\overline{\Omega} $ for each $j$ since $u_{\epsilon_{j}} \in C(\overline{\Omega})$.

Recall that
$$ F(D^{2}\psi,D\psi, \psi,x) =f-g \ge F(D^{2}u_{\epsilon_{j}},Du_{\epsilon_{j}},u_{\epsilon_{j}},x)  \quad \textrm{in} \ V_{j}.$$
And we also have $u_{\epsilon_{j}} =\psi$ on $\partial V_{j} \backslash \partial \Omega$.

Now we can apply Lemma \ref{obcom} to obtain $u_{\epsilon_{j}} \ge \psi $ in $V_{j}$,
which is a contradiction to the definition of $V_{j}$ and thus $V_{j} = \varnothing$ for each $j$.
Therefore, we can obtain $u \ge \psi $ in $ \overline{\Omega}$.

We next claim that 
$$ F(D^{2}u,Du,u,x)=f \qquad \textrm{in} \ V:= \{ x \in \Omega: u(x) > \psi(x) \}. $$
For each $m \in \mathbb{N}$, we have
$$ \Phi_{\epsilon_{j}}(u_{\epsilon_{j}}-\psi) \to 1  \quad \textrm{a.e. in} \ \bigg\{x \in \Omega: u(x)> \psi(x) +\frac{1}{m} \bigg\} $$
as $j \to \infty$.
Thus, for $$V= \{x\in \Omega : u(x) > \psi(x) \} = \bigcup_{m=1}^{\infty} \bigg\{x\in \Omega :  u (x)> \psi(x) + \frac{1}{m} \bigg\},$$ we derive
$$ g^{+} \Phi_{\epsilon_{j}}(u_{\epsilon_{j}}-\psi)+ f- g^{+}  \to f \qquad \textrm{a.e.} \ \ \textrm{in} \  V
$$
as $j \to \infty$.  
Thus, we deduce that
$$ F(D^{2}u,Du,u,x)=g^{+}+f-g^{+}= f \qquad \textrm{in} \ V$$ 
in the viscosity sense.

Therefore, we can conclude that $u$ is a viscosity solution of \eqref{eloblobs}.
Moreover, from \eqref{uepest} and \eqref{uepcon}, we have
\begin{align*}
 ||u||_{W^{2,p}(\Omega)}  \le \liminf_{j \to \infty} ||u_{\epsilon_{j}}||_{W^{2,p}(\Omega)}
\le  C(|| f ||_{ L^{p}(\Omega)}+||\psi ||_{ W^{2,p}(\Omega)})
\end{align*}
for some constant $C= C(n, \lambda, \Lambda, p, \delta_{0},b,c,  ||\beta||_{C^{2}(\partial \Omega)},  \diam(\Omega),\rho_{0})$.

For the uniqueness, let $u_{1}$ and $u_{2}$ be two viscosity solutions of \eqref{eloblobs}.
Suppose that $u_{1} \not\equiv u_{2}$.
Then we can assume without loss of generality that 
$$ G = \{ u_{2} > u_{1} \} \neq \varnothing. $$
Since $u_{2} > u_{1} \ge \psi$ in $G$, we see that $F(D^{2}u_{2},Du_{2},u_{2},x)= f$ in $G$ in the viscosity sense.
Then we have
\begin{align*}
\left\{ \begin{array}{ll}
F(D^{2}u_{1},Du_{1},u_{1},x) \le F(D^{2}u_{2},Du_{2},u_{2},x)=f & \textrm{in $G, $}\\
u_{1} =u_{2}  & \textrm{on $\partial G \backslash \partial \Omega$,} \\
\beta \cdot Du_{1} = \beta \cdot Du_{2} = 0 & \textrm{on $\partial G \cap \partial \Omega$.}\\
\end{array} \right. 
\end{align*}
Now applying \cite[Theorem 2.10]{MR1376656} or Lemma \ref{obcom} to $u_{1}-u_{2}$, we deduce that $u_{1} \ge u_{2}$ in $G$ whether $\partial G \cap \partial \Omega= \varnothing$ or not.
This contradicts the definition of the set $G$, and hence $u_{1}=u_{2}$. 
\end{proof}
\color{black}

\bibliographystyle{alpha}

\end{document}